\documentclass[11pt]{article}

\usepackage[margin=1in]{geometry}

\usepackage[utf8]{inputenc}
\usepackage{comment}
\usepackage{relsize}
\usepackage{hyperref}
\usepackage{tikz}
\usepackage{amssymb}
\usepackage{amsthm} 
\usepackage{mathtools}
\usepackage{lmodern}
\usepackage{verbatim, graphicx}
\usepackage{stmaryrd}
\usepackage{shuffle}
\hypersetup{
	colorlinks=true,
	linkcolor=blue,
	citecolor=blue}

\newtheorem{theorem}{Theorem}[section]
\newtheorem{corollary}[theorem]{Corollary}
\newtheorem{lemma}[theorem]{Lemma}

\numberwithin{equation}{section}

\theoremstyle{definition}

\newtheorem*{remark}{Remark}
\newtheorem*{remarks}{Remarks}
\newtheorem{example}[theorem]{Example}
\newtheorem{examples}[theorem]{Examples}

\newcommand{\cK}{\mathcal{K}}
\newcommand{\cP}{\mathcal{P}}

\newcommand{\Irr}{\mathrm{Irr}}

\newcommand{\Ind}{\mathrm{Ind}}

\newcommand{\dd}{\displaystyle}

\newcommand{\cN}{\mathcal{N}}
\newcommand{\cI}{\mathcal{I}}
\newcommand{\FF}{\mathbb{F}}

\newcommand{\f}{\mathrm{f}}
\newcommand{\Cl}{\mathtt{Cl}}
\newcommand{\Ch}{\mathtt{Ch}}
\newcommand{\tS}{\mathsf{S}}
\newcommand{\CC}{\mathbb{C}}

\newcommand{\GL}{\mathrm{GL}}

\newcommand{\cL}{\mathcal{L}}

\newcommand{\Res}{\mathrm{Res}}

\newcommand{\spanning}{\text{-span}}

\newcommand{\One}{1\hspace{-.13cm}1}
\newcommand{\cA}{\mathcal{A}}
\newcommand{\cB}{\mathcal{B}}
\newcommand{\cC}{\mathcal{C}}
\newcommand{\cM}{\mathcal{M}}
\newcommand{\gco}{\mathrm{gco}}

\newcommand{\lcm}{\mathrm{lcm}}

\newcommand{\anti}{\mathrm{AntCh}}

\newcommand{\subsp}{\mathrm{subsp}}
\newcommand{\bl}{\mathrm{bl}}

\newcommand{\Hom}{\mathrm{Hom}}

\title{The structure of normal lattice \\ supercharacter theories}

\date{}
\author{Farid Aliniaeifard and Nathaniel Thiem}

\begin{document}
	
	\maketitle

\begin{abstract}
The character theory of finite groups has numerous basic questions that are often already quite involved: enumerating of irreducible characters, their character formulas, point-wise product decompositions, and restriction/induction between groups.  A supercharacter theory is a framework for simplifying the character theory of a finite group, while ideally not losing all important information.  This paper studies one such theory that straddles the gap between retaining valuable group information while reducing the above fundamental questions to more combinatorial lattice constructions.   
\end{abstract}

\section{Introduction} \label{Introduction}

Through the work of Andr\'e \cite{And95} and Yan \cite{Yan01}, supercharacter theory has allowed us to apply the tools of character theory to groups previously deemed intractable.  The more general framework developed by Diaconis and Isaacs \cite{DI08} further fleshes out a theory that can be adapted to different characteristics one might wish to emphasize in groups (eg. if one wants to study real-valued characters).   However, the construction and existence of supercharacter theories remains somewhat mysterious.  There are some basic techniques that apply to all groups, and in some surprising cases (eg. $\mathrm{Sp}_6(2)$ \cite{BLLW}) these give a complete understanding.  Nevertheless, in most cases we do not have a good understanding of what supercharacter theories are possible.    

In his thesis work, Aliniaeifard \cite{AL17} developed an alternate approach that centers on the set of normal subgroups.  A standard result in character theory is that knowledge of the characters identifies all normal subgroups of the group; these form a lattice under inclusion.  It turns out that every supercharacter theory identifies a sublattice of normal subgroups, and this naturally partitions supercharacter theories according to the sublattices they ``see."   Aliniaeifard identified the unique coarsest supercharacter theory corresponding to each sublattice, and identified numerous desirable characteristics exhibited by this theory.  

This point of view naturally leads to a notion of ``simple" supercharacter theory, or one which only identifies the trivial subgroup and the whole group (eg. simple groups only have simple supercharacter theories).  Burkett \cite{Bu} showed that there is a Jordan--H\"older type factorization for supercharacter theories into simple supercharacter theories, and has developed a framework for super versions of various chains of normal subgroups.

This paper examines the notion of a normal lattice supercharacter theory a bit more closely, tying the character theory to the underlying lattice in a more explicit fashion.  There are numerous problems in character theory that generally have difficult solutions, such as explicit character formulas for the irreducible characters, decompositions of tensor products (eg. Kronecker products as in \cite{BB}) and restrictions between groups, etc.  It turns out that many of these problems have elegant solutions in the case of normal lattice supercharacter theories, and this paper would like to make the case that while these theories  are  non-trivial, they retain high levels of computability.

We begin with a review of the necessary lattice notation and an introduction to supercharacter theories.  We then proceed to Corollary \ref{CharacterFormula} which gives an explicit multiplicative character formula for the supercharacters, reminiscent of the combinatorial character formulas found in \cite{Yan01} for the maximal unipotent subgroups of $\GL_n(q)$.  Under some additional hypotheses that guarantee some level of generality, we also give a decomposition of the point-wise product in Corollary \ref{TensorProduct}.

In the case that the normal subgroups form a distributive lattice, we can say even more.  An advantage of the normal lattice supercharacter theories is that they are somewhat more canonical, analogous to how every group has a partition by conjugacy classes.  This feature allows us to better compare supercharacter theories of groups related via inclusion.  We conclude with a description of restriction between finite groups $H\subseteq G$ in Theorem \ref{RestrictionTheorem}, where we explicitly decompose the restricted supercharacter theory in terms of the supercharacter theory of the subgroup.  This result includes some natural compatibility conditions suggested by the underlying theory.

The original motivation for this work comes out of the supercharacter theory of non-nesting partitions for pattern groups \cite{An15}.  However, to keep this paper in a more manageable form, we have separated the applications to these groups into a companion paper \cite{AT18}.  In this paper we instead illustrate the theory with several families of abelian groups including cyclic groups and vector spaces.  

\section{Preliminaries} 

This section will fix some of the standard notation on lattices in Section \ref{Lattices}, and in Section \ref{SupercharacterTheory} we review both the definition of supercharacter theories and the construction of the normal lattice supercharacter theory as defined by \cite{AL17}.

\subsection{Lattices of normal subgroups} \label{Lattices}

The goal of this section is to fix notation for lattices with a particular emphasis on lattices of normal subgroups of a finite group.  

A \textbf{\emph{lattice}} $\cL$ is a poset such that for each pair $(K,L)\in \cL\times \cL$ there is a unique least upper bound $K\vee L$ and a unique greatest lower bound $K\wedge L$.  We say an element $L\in \cL$ \textbf{\emph{covers}} $K\in \cL$ if $L$ is minimal with the property that $K\prec_{\cL} L$ and $K\neq L$.  Given $L\in \cL$, let
\begin{equation}
\cC(L)= \{O\in \cL\mid O\text{ covers } L\}.
\end{equation}

Given a poset $\cP$, let
\begin{equation*}
J^\vee(\cP)=\{I\subseteq \cP\mid i \in I, j\in \cP, i\prec_\cP j\text{ implies } j\in I\}
\end{equation*}
be the lattice of co-ideals,  ordered by inclusion.

A \textbf{\emph{distributive lattice}} $\cL$ is a lattice such that for $K,L,M\in \cL$, 
$$K\vee(L\wedge M)=(K\vee L)\wedge (K\vee M).$$
An element $K\in \cL$ is \textbf{\emph{meet irreducible}} if $|\cC(K)|= 1$.   
The fundamental theorem for finite distributive lattices says that $\cL\cong J^\vee(\cM)$ where $\cM$ is the subposet of meet irreducible elements of $\cL$.    This implies that given any $K\in \cL$ there exists a unique antichain $\cA$ of $\cM$ such that 
$$K=\bigwedge_{A\in \cA} A.$$ 
Define
$$\anti(\cL)=\{\text{anti-chains in $\cL$}\}.$$

Our main example of lattices are sublattices of normal subgroups of a finite group $G$.  Specifically, let
$$\ker(G)=\{N\triangleleft G\}$$
be the lattice ordered by inclusion.  In this case, given $M,N\in \ker(G)$,
$$M\vee N=MN\qquad \text{and} \qquad M\wedge N=M\cap N.$$
A \textbf{\emph{sublattice}} $\cN$ of $\ker(G)$ will be a subset such that 
\begin{enumerate}
\item[(L1)]  $\{1\}, G\in \cN$,
\item[(L2)]  $M,N\in \cN$ implies $MN,M\cap N\in \cN$.
\end{enumerate}
We have that by (L2), every sublattice is modular, so if $M\subseteq N$, then
$$ML\cap N= M(L\cap N).$$

For a subset $\cA\subseteq \ker(G)$, let
\begin{equation}
\overline{\cA} = \prod_{N\in \cA} N \qquad \text{and}\qquad
\underline{\cA} = \bigcap_{N\in \cA} N. \label{OverUnderline}
\end{equation}

\begin{examples}\hfill

\vspace{.25cm}

\noindent\textbf{Cyclics.} Let $C_n=\langle x\rangle$ be a cyclic group.  Then $\ker(C_n)$ is isomorphic to the lattice of divisors of $n$.  For example,
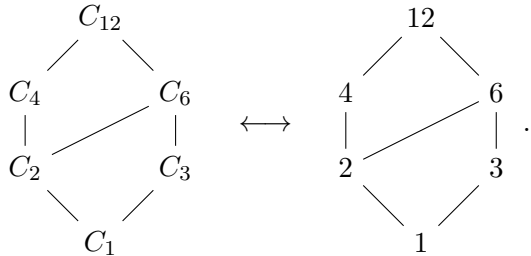
\begin{equation*}\begin{tikzpicture}[baseline=1.5cm]
\foreach \w/\x/\y/\z in {1/0/0/{C_1},2/-1/1/{C_2},3/1/1/{C_3},4/-1/2/{C_4}, 6/1/2/{C_6}, 12/0/3/{C_{12}}}
	\node (\w) at (\x,\y) {$\z$};
\draw (1) -- (2) -- (4) -- (12) -- (6) -- (3) -- (1);
\draw (2) -- (6);
\end{tikzpicture}\quad \longleftrightarrow \quad 
\begin{tikzpicture}[baseline=1.5cm]
\foreach \w/\x/\y/\z in {1/0/0/{1},2/-1/1/{2},3/1/1/{3},4/-1/2/{4}, 6/1/2/{6}, 12/0/3/{12}}
	\node (\w) at (\x,\y) {$\z$};
\draw (1) -- (2) -- (4) -- (12) -- (6) -- (3) -- (1);
\draw (2) -- (6);
\end{tikzpicture}\ .
\end{equation*}
 In general,
$$C_a\cap C_b=C_{\gcd(a,b)}\quad \text{and}\quad C_a C_b=C_{\lcm(a,b)},$$
and $\ker(C_n)$ is a distributive lattice with meet irreducible elements
\begin{equation}\label{CyclicIrreducibles}
\{C_m\mid m<n,n/m\text{ is a prime power}\}.
\end{equation}

\vspace{.5cm}

\noindent\textbf{Vector spaces.}  Let $\FF_q$ be the finite field with $q$ elements, and $V$ an $\FF_q$-module.  Then $V$ is a finite (elementary) abelian group.  While the usual character theory of $V$ sees all the normal subgroups of $V$, we would prefer some sublattices that respect the vector space structure.  Thus,
$$\subsp(V)=\{U\subseteq V\mid U\text{ an $\FF_q$-submodule}\},$$
 and for a fixed basis $\cB\subseteq V$,
 $$\subsp_\cB(V)=\{\FF_q\spanning\{a\in\cA\}\mid \cA\subseteq \cB\},$$
 which is isomorphic to the lattice of subsets of $\cB$.
 
 The lattice $\subsp(V)$ is not distributive in general, but $\subsp_\cB(V)$ is distributive with meet irreducible elements
 $$\{\FF_q\spanning\{a\in \cB\mid a\neq b\}\mid b\in \cB\}.$$
 \end{examples}

\subsection{Supercharacter theories} \label{SupercharacterTheory}

Supercharacter theories give us a framework for simplifying the character theory of a groups while maintaining the representation theoretic underpinnings.  While one can view them as central Schur rings, this section outlines the more representation theoretic framework introduced by Diaconis--Isaacs \cite{DI08}.

Given a set partition $\cK$ of $G$, let
$$\f(G;\cK)=\{\psi:G\rightarrow \CC\mid \{g,h\}\subseteq K\in \cK\text{ implies } \psi(g)=\psi(h)\}$$
be the set of functions constant on the blocks of $\cK$.

A \textbf{\emph{supercharacter theory}} $\tS$ of a finite group $G$ is a pair $(\Cl(\tS),\Ch(\tS))$ where $\Cl(\tS)$ is a set partition of $G$ and $\Ch(\tS)$ is a set partition of the irreducible characters $\Irr(G)$ of $G$, such that
\begin{enumerate}
\item[(SC1)] $\{1\}\in \Cl(\tS)$,
\item[(SC2)] $|\Cl(\tS)|=|\Ch(\tS)|$,
\item[(SC3)] For each $X\in \Ch(\tS)$,
$$\sum_{\psi\in X}\psi(1)\psi\in \f(G;\Cl(\tS)).$$
\end{enumerate}  

We refer to the blocks of $\Cl(\tS)$ as the \textbf{\emph{superclasses}} of $\tS$, and the elements of 
$$\{\sum_{\psi \in X}\psi(1)\psi\mid X\in \Ch(\tS) \}$$
 as \textbf{\emph{supercharacters}} of $\tS$.   In fact, the supercharacters of $\tS$ will form an orthogonal basis for $\f(G;\Cl(\tS))$; in particular, the superclasses are unions of conjugacy classes. 
 
The trivial examples of supercharacter theories have partitions 
\begin{align*} (\Cl,\Ch)&=\Big(\{\text{conjugacy classes}\},\{\{\psi\}\mid \psi\in \Irr(G)\Big)\\
(\Cl,\Ch)&=\Big(\{\{1\},G-\{1\}\},\{ \{\One\} ,\Irr(G)-\{\One\}\Big),
\end{align*}
where $\One$ is the trivial character of $G$.  

Given a group $G$ with a supercharacter theory $\tS$, we call a normal subgroup \textbf{\emph{$\tS$-normal}} if it is a union of superclasses.  These normal subgroups can also be defined as intersections of kernels of the corresponding supercharacters \cite{M11}, and the set
$$\ker(\tS) = \{N\triangleleft G\mid N\text{ a union of superclasses}\}$$
forms a sublattice of all normal subgroups \cite{AL17}.   Note that the trivial supercharacter theory of conjugacy classes gives all normal subgroups $\ker(G)$.   More generally, given an arbitrary sublattice of normal subgroups of $G$, one might ask which supercharacter theories (if any) see at least those normal subgroups.  The following supercharacter theory defined in \cite{AL17} gives the coarsest such supercharacter theory for each sublattice of normal subgroups.

\begin{theorem}[\cite{AL17}]\label{NormalLatticeSupercharacterTheory}
 Let $\cN\subseteq \ker(G)$ be a sublattice.   
 \begin{enumerate}
 \item[(a)] The partitions
$$\Cl(\tS(\cN))=\{N_\circ\neq \emptyset\mid N\in \cN\},\quad \text{where} \quad N_\circ=\{g\in N\mid g\notin M\in \cN, \text{ if $N\in \cC(M)$}\},$$
and 
$$\Ch(\tS(\cN))=\{X^{N^\bullet}\neq \emptyset\mid N\in \cN\},\quad \text{where} \quad  X^{N^\bullet}=\{\psi\in \Irr(G)\mid N\subseteq \ker(\psi)\nsupseteq O\in \cC(N)\}$$
define a supercharacter theory $\tS(\cN)$ with $\cN\subseteq \ker(\tS(\cN))$.
\item[(b)] If $\tS$ is any other supercharacter theory of $G$ with $\cN\subseteq \ker(\tS)$, then $\ker(\tS(\cN))\subseteq \ker(\tS).$
\end{enumerate}
\end{theorem} 
Note that \cite{AL17} also provides recursive supercharacter formulas for the supercharacters
$$\chi^{N^\bullet}=\sum_{\psi\in X^{N^\bullet}} \psi(1)\psi,$$
and proves a number of nice properties (eg. these supercharacters are integer valued).  Just as a normal subgroup is a union of superclasses, for each normal subgroup $N\triangleleft G$ there is a natural character 
\begin{equation}\label{SubgroupCharacterDecomposition}
\chi^N=\sum_{O\supseteq N} \chi^{O^\bullet}=\sum_{\psi\in \Irr(G)\atop N\subseteq \ker(\psi)} \psi(1)\psi
\end{equation}
that is a ``union" of supercharacters with character formula
\begin{equation}\label{SubgroupCharacterFormula}
\chi^N(g)=\left\{\begin{array}{@{}ll}
|G/N| & \text{if $g\in N$,}\\
0 & \text{otherwise.}
\end{array}\right.
\end{equation}

\begin{examples}\hfill

\vspace{.25cm}

\noindent\textbf{Cyclics.}
Let $C_n=\langle x\rangle$ be a cyclic group.  Then for $d\mid n$,
\begin{align*}
(C_d)_\circ &=\{x^{jn/d}\mid  \text{order of $x^{jn/d}$ is $d$}\},\\
X^{C_d^\bullet} &=\{\psi:C_n\rightarrow \CC^\times\text{ homomorphism}\mid \text{order of $\psi(x)$ is $n/d$}\}.
\end{align*}
Note that in this supercharacter theory, superclasses and supercharacters are indexed by divisors of $n$, so for $d\mid n$ write 
$$\chi^d=\chi^{C_d^\bullet}.$$
Note that this supercharacter theory was used (though not defined in this way) to study Ramanujan sums in \cite{FGK}.

\vspace{.5cm}

\noindent\textbf{Vector spaces.}  For the lattice $\subsp(V)$, we have that for $U\in \subsp(V)$
$$U_\circ=\left\{\begin{array}{ll} \{0\}  & \text{if $\dim(U)=0$,}\\ U-\{0\} & \text{if $\dim(U)=1$},\\
\emptyset & \text{otherwise.}\end{array}\right.$$
and 
\begin{align*}
X^{U^\bullet}&=\left\{\begin{array}{ll} \{ \One \} & \text{if $U=V$,}\\
\{\psi\in \Hom(U,\FF_q^+)\mid U=\ker(\psi)\} & \text{if $\dim(U)=\dim(V)-1$,}\\ 
\emptyset & \text{otherwise.}\end{array}\right.
\end{align*}
In this case, the superclasses and supercharacters are indexed by the number of one and zero dimensional spaces of $V$ or 
$$1+\frac{q^n-1}{q-1}$$
if $\dim(V)=n$.

On the other hand, for the lattice $\subsp_\cB(V)$, we have that for $A\subseteq \cB$
$$\FF_q\spanning\{a\in A\}_\circ=\Big\{\sum_{a\in A} c_a a\mid c_a\neq 0, a\in A\Big\},$$
and 
$$X^{\FF_q\spanning\{a\in A\}^\bullet}=\{\psi\in  \Hom(U,\FF_q^+)\mid \ker(\psi)\cap \cB=A\}.$$
 In this case, the superclasses and supercharacters are indexed by subsets of $\cB$, so for $A\subseteq \cB$ write
$$\chi^A=\chi^{\FF_q\spanning\{a\in A\}^\bullet}.$$
\end{examples}

\section{Normal lattice theories}

This section includes our main results.  We begin in Section \ref{Supermodules} by constructing the dual lattice of $\ker(\tS)$ in terms $G$-modules.  This gives a recursive construction for the modules corresponding to the supercharacters in $\tS$.  We use the module structure to deduce several degree sum results in Section \ref{CharacterFormulas}, allowing us also to obtain an explicit character formula.  In Section \ref{TensorProducts} we find a decomposition for the point-wise product in some situations, and we conclude in Section \ref{RestrictionFormula} with a decomposition of the restriction of supercharacters under some beneficial assumptions.

\subsection{Supermodules}\label{Supermodules}

Given a normal subgroup $N\subseteq G$, we obtain the permutation module
$$\Ind_N^G(\One)\cong \CC G\otimes_{\CC N} \One\cong \CC G e_N,\qquad \text{where $e_N=\frac{1}{|N|}\sum_{n\in N} n$.}$$
From the last isomorphism, we may view each of these modules as submodules of the regular module $\CC G$ (the permutation module coming from the normal subgroup $\{1\}$).    Note that by (\ref{SubgroupCharacterFormula}),
$$\chi^N=\mathrm{tr}(\cdot, \CC G e_N).$$

Given a lattice of normal subgroups $\ker(\tS)$ we obtain a corresponding lattice of modules 
$$\ker(\tS)^\vee=\{\CC G e_N\mid N\in \ker(\tS)\}$$
ordered by containment of modules.  Since $M\subseteq N$ implies 
$$e_Ne_M=e_N$$
we have $\CC Ge_N\subseteq \CC Ge_M$.  On the other hand, if $\CC Ge_N\subseteq \CC  Ge_M$, then $e_{MN}=e_Me_N=e_N$, so $M\subseteq N$.   Thus,
$$\CC Ge_N\vee \CC Ge_M=\CC Ge_{N\cap M}\qquad \text{and}\qquad \CC Ge_N\wedge \CC Ge_M=\CC Ge_{NM}.$$
Define $G^{M^\bullet}$ by
\begin{align*}
\CC G e_M &\cong \Big(\sum_{N\supset M} \CC G e_N\Big)\oplus G^{M^\bullet}\\
&=\Big(\sum_{N\text{ covers } M} \CC G e_N\Big)\oplus G^{M^\bullet}.
\end{align*}
Then by the character decomposition formula (\ref{SubgroupCharacterDecomposition}) of $\chi^M$, 
$$\chi^{M^{\bullet}}=\mathrm{tr}(\cdot, G^{M^\bullet}).$$
In particular, taking dimensions, we obtain 
$$\chi^{M^{\bullet}}(1)=|G/M|-\dim\Big( \sum_{M\text{ covers } N} \CC G e_N\Big).$$ 
The next result gives a better understanding of modules that arise in this way.  Given a subgroup $N\subseteq G$, the notation $\widehat{G/N}$ will denote a transversal for the cosets with $1\in  \widehat{G/N}$.  Also, recall (\ref{OverUnderline}) for definition of $\overline{\cA}$.

\begin{lemma} \label{DecompositionLemma}
Let $\ker(\tS)$ be a lattice for a supercharacter theory of $G$.  Fix $M\in \ker(\tS)$  and let $\cA$ be an antichain in the interval  $[M,G]$ that satisfies $\overline{\cA-O}\neq \overline{\cA}$ for all $O\in \cA$.  
Let
$$\CC G e_M=\sum_{N\in \cA} \CC G e_{N} \oplus V.$$
Then 
\begin{enumerate}
\item[(a)] There exists a choice of transversals $\widehat{N/M}$ and $\widehat{G/\overline{\cA}}$  such that
$$V=\CC\spanning\{g\prod_{N\in \cA}(1-b_N)e_M\mid g\in \widehat{G/\overline{\cA}}, b_N\in \widehat{N/M}-\{1\}\}.$$
\item[(b)] If $\overline{\cA-O}\cap O=M$ for all $O\in \cA$, then
$$\dim(V)=|G/\overline{\cA}| \prod_{N\in \cA} (|N/M|-1).$$
\end{enumerate}
\end{lemma}
\begin{proof} For the proof, order the elements of $\cA=\{N_1,N_2,\ldots N_\ell\}$.

(a)
Fix coset representatives $\widehat{G/M}$ with $1\in \widehat{G/M}$  and for each $1\leq j\leq \ell$ let 
\begin{enumerate} 
\item[(a)] $\widehat{N_j/M}\subseteq \widehat{G/M}$,
\item[(b)] $\widehat{G/N_1\cdots N_j}$ be equivalence class representatives the pre-image equivalence relation in $\widehat{G/M}$ coming from the canonical projection $\pi_{G/N_1\cdots N_j}$. 
\end{enumerate} 
Choose the representatives in (b) such that $1\in \widehat{G/N_1\cdots N_j}$ and 
$$\widehat{G/N_1\cdots N_j}\supseteq \widehat{G/N_1\cdots N_{j+1}}.$$
Define $V_1$ by
$$\CC G e_M=\CC G e_{N_1}\oplus V_1.$$
The natural basis of $\CC G e_M$ is
$$\{gn_1e_M\mid g\in \widehat{G/N_1}, n_1\in \widehat{N_1/M}\}.$$
Since 
$$\CC Ge_{N_1}=\CC\spanning\{ge_{\widehat{N_1/M}} e_M\mid g\in \widehat{G/N_1}\},$$
we have 
\begin{align*}
V_1 &=\Big\{\sum_{g\in \widehat{G/N_1}\atop n_1\in \widehat{N_1/M}} c_{g,n_1} gn_1e_M\mid \sum_{n_1\in \widehat{N_1/M}} c_{g,n_1}=0, g\in \widehat{G/N_1}\}\\
&=\CC \spanning\{ g(1-n_1)e_M\mid g\in \widehat{G/N_1}, n_1\in \widehat{N_1/M}-\{1\}\}.
\end{align*}
Define $W_j$ and $V_j$ by
$$V_{j-1}=W_j\oplus V_j.$$
where 
$$W_j=\CC\spanning\{g(1-n_1)(1-n_2)\cdots(1-n_{j-1})e_{\widehat{N_j/M}}e_M\mid g\in \widehat{G/N_1\cdots N_j},n_i\in \widehat{N_i/M}-\{1\}\},$$
and
$$V_j=\CC\spanning\{g(1-n_1)(1-n_2)\cdots(1-n_j)e_M\mid g\in \widehat{G/N_1\cdots N_j},n_i\in \widehat{N_i/M}-\{1\}\}.$$
Then
$$\CC G e_M=W_1\oplus W_2\oplus \cdots \oplus W_\ell \oplus V_\ell.$$
Note that 
\begin{align*}
\CC G e_{N_j} &=W_1e_{N_j}\oplus W_2e_{N_j}\oplus \cdots \oplus W_\ell e_{N_j}\oplus V_\ell e_{N_j}\\
&= W_1e_{N_j}\oplus W_2e_{N_j} \oplus \cdots \oplus W_j e_{N_j}\\
&\subseteq W_1\oplus \cdots \oplus W_\ell.
\end{align*}
Thus,
$$\sum_{j=1}^\ell \CC G e_{N_j}\subseteq W_1\oplus \cdots \oplus W_\ell.$$
Conversely, for each $W_j$ there exists $e_{N_j}$ such that
$$W_je_{N_j}=W_j.$$
Thus,
$$\sum_{j=1}^\ell \CC G e_{N_j}\supseteq W_1\oplus \cdots \oplus W_\ell$$
and
$$V_\ell=V.$$

(b) If $\overline{\cA-O}\cap O=M$ for all $O\in \cA$, then
$$\{g(1-n_1)(1-n_2)\cdots(1-n_\ell)e_M\mid g\in \widehat{G/\overline{\cA}}, n_j\in \widehat{N_j/M}-\{1\}\}$$
is a basis, and so we get the corresponding dimension formula.
\end{proof}

\begin{remarks} \hfill

\begin{enumerate}
\item[(1)] The key idea of this lemma is that modules allow us to express overlap between modules in a way that characters are ill equipped to do. \item[(2)] The main case we are interested in applying Lemma \ref{DecompositionLemma} is when $\cA\subseteq \cC(M)$.  In this case, the condition $\overline{\cA-\{O\}}\cap O=M$ for all $O\in \cA$ is true if and only if $\overline{\cA-\{O\}}\neq \overline{\cA}$ for $O\in \cA$.     If this condition holds for all $O\in A$, then we say $\cA$ is in \textbf{\emph{general position}} over $M$.

In fact, if we are in a distributive lattice, then distributivity implies  $\overline{\cA-\{O\}}\cap O=M$ for all $O\in \cA$ is always true when $\cA\subseteq \cC(M)$.
\end{enumerate}
\end{remarks}

\subsection{Character formulas}\label{CharacterFormulas}

This section works out some general character formulas that are mostly direct consequences of Lemma \ref{DecompositionLemma}.  The first result examines a general character degree sum that carves out a piece of a lattice of normal subgroups $\ker(\tS)$.  

\begin{theorem}\label{DegreeSumTheorem}
 Let $\ker(\tS)\subseteq \ker(G)$ be a lattice of normal subgroups.  Let $K,L,M\in \ker(\tS)$. Assume the set
$$\cC_L^\perp(K,M) = \{O\in \cC(KM)\mid O\cap L\neq K\}$$
is in general position over $KM$.
Then
$$\sum_{N\supseteq M\atop N\cap L=K} \chi^{N^{\bullet}}(1)=\left\{\begin{array}{ll} 
|G/KM| & \text{if $\cC_L^\perp(K,M)=\emptyset$,}\\
\dd |G/\overline{\cC_L^\perp(K,M)}| \prod_{O\in \cC_L^\perp(K,M)}(|O/KM|-1) & \text{if $KM\cap L=K$, $\cC_L^\perp(K,M)\neq \emptyset$,}\\
0 & \text{if $KM\cap L\neq K$.}
\end{array}\right.$$
\end{theorem}

\begin{proof} Note that $KM$ is the minimal normal subgroup $N$ containing $M$ that satisfies $N\cap L\supseteq K$.  If $KM\cap L\neq K$, then the sum is therefore empty and we get 0.  If $KM\cap L= K$ and $\cC_L^\perp(K,M)=\emptyset$, then by (\ref{SubgroupCharacterFormula}), we get $|G/KM|$.  

Next, suppose $KM\cap L= K$ and $\cC_L^\perp(K,M)\neq \emptyset$.   Suppose $O\supseteq KM$ is minimal such that $O\cap L\neq K$.  We want to show that $O\in \cC_L^\perp(K,M)$.  We have
$$
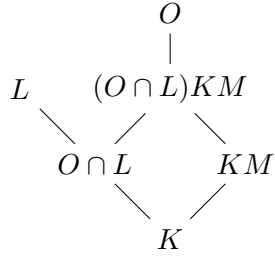
\begin{tikzpicture}
\foreach \x/\y/\z in {0/2/L,1/1/{O\cap L},2/0/K,2/2/{(O\cap L)KM}, 3/1/{KM},2/3/O}
	\node (\x\y) at (\x,\y) {$\z$};
\draw (02) -- (11) -- (20) -- (31) -- (22) -- (23);
\draw (11) -- (22);
\end{tikzpicture}
$$
 Since $(O\cap L)KM\supseteq KM$ and $L\cap \Big((O\cap L)KM\Big)=O\cap L$, the minimality of $O$ implies $O=(O\cap L)KM$.  However, by modularity, if there exists $KM\subset N\subset O$ then $N\cap L\neq K$.  Thus, the minimality of $O$ implies $O$ covers $KM$.

We have 
\begin{equation*}
V=\bigoplus_{N\supseteq M\atop N\cap L=K}  G^{N^{\bullet}}  =\bigoplus_{N\supseteq KM\atop N\cap L=K}  G^{N^{\bullet}}
\end{equation*}
We therefore apply Lemma \ref{DecompositionLemma}(a) to
$$\CC G e_{KM}=\sum_{O\in \cC_L^\perp(K,M)} \CC G e_{O} \oplus V.$$

We have that  $\overline{\cC_L^\perp(K,M)-O}\cap O\in \{O,KM\}$ since $O$ is a cover of $KM$.  By general position $\overline{\cC_L^\perp(K,M)-O}\neq \overline{\cC_L^\perp(K,M)}$, so $O\nsubseteq \overline{\cC_L^\perp(K,M)-O}$.  Thus, we may apply Lemma \ref{DecompositionLemma}(b) to get the desired dimension formula.
\end{proof}

A special case of Theorem \ref{DegreeSumTheorem} gives us the degrees of the supercharacters of a normal lattice supercharacter theory.

\begin{corollary}  \label{SuperDegreeFormula} Let $\ker(\tS)\subseteq \ker(G)$ be a sublattice.  For $M\in \ker(\tS)$ with $\cC(M)$ in general position over $M$,
$$\chi^{M^\bullet}(1)=|G/\overline{\cC(M)}| \prod_{N\in \cC(M)} (|N/M|-1).$$
\end{corollary}
\begin{proof}
This is the case where $L=G$ and $K=M$ in Theorem \ref{DegreeSumTheorem}.
\end{proof}

Since this situation satisfies the hypotheses of Lemma \ref{DecompositionLemma} (b), we have an explicit basis and can use this basis to compute the trace of the module.

\begin{corollary}\label{CharacterFormula}
 Let $\ker(\tS)\subseteq \ker(G)$ be a sublattice.   For $M\in \ker(\tS)$ with $\cC(M)$ in general position over $M$ and $g\in G$,
$$\chi^{M^\bullet}(g)=\left\{
\begin{array}{@{}ll@{}}
\dd \chi^{M^\bullet}(1)\prod_{N\in \cB} \frac{1}{ (1-|N/M|)}  & \text{if $g\in \overline{\cB}$, for $\cB=\min\{B\subseteq \cC(M)\mid g\in \overline{B}\}$,}\\
 0 & \text{if $g\notin \overline{\cC(M)}$}.
\end{array}\right.$$
\end{corollary}

\begin{proof}
Suppose $g\in G$, then
\begin{align*}\chi^{M^\bullet}(g)&=\mathrm{tr}(g,G^{M^\bullet})\\
&=\hspace{-.5cm}\sum_{h\in \widehat{G/\overline{\cC(M)}}\atop  a_N\in\widehat{N/M}-\{1\},N\in \cC(M)}\hspace{-.5cm} \mathrm{Coeff}\Big(gh\prod_{N\in \cC(M)} (1-a_N)e_M;h\prod_{N\in \cC(M)} (1-a_N)e_M\Big).
\end{align*}
If $g\notin \overline{\cC(M)}$, then $gh \overline{\cC(M)}\neq h \overline{\cC(M)}$, so
$$\mathrm{Coeff}\Big(gh\prod_{N\in \cC(M)} (1-a_N)e_M;h\prod_{N\in \cC(M)} (1-a_N)e_M\Big)=0$$
for all $h,a_N$.  If $g\in M$, then since $M$ is normal,
\begin{align*}
gh\prod_{N\in \cC(M)} (1-a_N)e_M
&=ge_Mh\prod_{N\in \cC(M)} (1-a_N)\\
&=e_Mh\prod_{N\in \cC(M)} (1-a_N)\\
&=h\prod_{N\in \cC(M)} (1-a_N)e_M.
\end{align*}
Suppose $g\in \overline{\cB}$ with $\cB$ nonempty and minimal.  If $N,N'\in \cB$ are not equal, then since they are covers
$$[N,N']\subseteq N\cap N' = M,$$
so for $a_N\in N$, $a_{N'}\in N'$, we have $a_Na_{N'}a_N^{-1}\in a_{N'}M$, and
$$a_N(1-a_{N'})e_M=a_N(1-a_{N'})a_N^{-1}a_Ne_M=(1-a_{N'})e_Ma_N.$$
Write $h^{-1}gh=\prod_{N\in \cC(M)}g_N$ with $g_N\in N$ and $g_N=1$ if $N\notin\cB$. Then 
\begin{align*}
gh\prod_{N\in \cC(M)} (1-a_N)e_M
&= h(\prod_{N\in \cC(M)} (g_N-g_Na_N)e_M\\
&=h \prod_{N\in \cC(M)} \Big((1-g_Na_N)-(1-g_N)\Big) e_M.
\end{align*}
Thus, we get a coefficient if for each $N\in \cC(M)$ either $g_N=1$ or $g_N=a_N$.  In fact, in this case
\begin{align*}
\sum_{h\in \widehat{G/\overline{\cC(M)}}\atop  a_N\in\widehat{N/M}-\{1\},N\in \cC(M)}\hspace{-.5cm} \mathrm{Coeff}\Big(gh\prod_{N\in \cC(M)} (1-a_N)e_M;h\prod_{N\in \cC(M)} (1-g_N)e_M\Big)
&=\sum_{h\in \widehat{G/\overline{\cC(M)}}\atop g_N\in \widehat{N/M}-\{1\}, N\notin \cB}  (-1)^{|\cB|}\\
&=\frac{\chi^{M^\bullet}(1)}{\prod_{N\in \cB} (|N/M|-1)}  (-1)^{|\cB|},
\end{align*}
as desired. 
\end{proof}

\begin{examples} \hfill

\vspace{.25cm}

\noindent\textbf{Cyclics.}  Consider $\ker(C_n)$ and fix  $b\mid n$. Let
\begin{equation}\label{CyclicCovers}
\cC(C_b)=\{C_{pb}\mid p\in P\}\quad \text{where}\quad P_b=\{p\text{ prime} \mid pb\mid n\},
\end{equation}
and for $O\subseteq P_b$, let $\overline{O}=\prod_{p\in O}p$.
By Corollary \ref{CharacterFormula},
$$\chi^b(x^{n/a})=\left\{\begin{array}{ll}
\dd \frac{n}{b}\prod_{p\in O} \Big(\frac{-1}{p}\Big)\prod_{p\in P_b- O} \Big(1-\frac{1}{p}\Big) & \text{if $O\subseteq P_b$ minimal exists so $a\mid b\overline{O}$, }\\
0 & \text{otherwise.}
\end{array}\right.$$

\vspace{.5cm}

\noindent\textbf{Vector spaces.}  For $\subsp(V)$, since if $\dim(U)=\dim(V)-1$, $\cC(U)=\{V\}$,
$$\chi^{U^\bullet}(v)=\left\{\begin{array}{ll}q-1 & \text{if $v\in U$,}\\ -1 & \text{if $v\notin U$}. \end{array}\right.$$
For $\subsp_\cB(V)$, given $A\subseteq \cB$, for each $b\in \cB-A$, $b\in W$ for some $W\in \cC(\FF_q\spanning\{a\in A\})$, so
$$\overline{\cC(\FF_q\spanning\{a\in A\})}=V.$$
Thus, for $D\subseteq \cB$
$$\chi^{A}\Big(\sum_{d\in D}d\Big)=(q-1)^{|\cB-(A\cup D)|} (-1)^{|(\cB-A)\cap D|}. 
$$
\end{examples}

In \cite[Theorem 3.4]{AL17}, Aliniaeifard gives a lattice theoretic formula for the character values given by
$$\chi^{N^\bullet}(g)=\sum_{O\subseteq N\atop g\in O} \mu(N,O)\frac{|G|}{|O|}.$$
In the examples of $\ker(C_n)$ and $\subsp_\cB(V)$ above, the M\"obius functions for the divisor lattice and subset lattice are well-known, so we get some identities for free. 
For a set of distinct primes $P_b$ and $a\mid \overline{P_b}$,
$$\sum_{A\subseteq P_b\atop a\mid b\overline{A}}\frac{(-1)^{|A|}}{\overline{A}}= \prod_{p\in O}\Big(\frac{-1}{p}\Big)\prod_{p\in P_b- O} \Big(1-\frac{1}{p}\Big), \quad \text{where}\quad \text{$O\subseteq P_b$ is minimal with $a\mid b\overline{O}$.}$$
where we cancel the $n/b$ that appears on both sides.  In the second example, for $D\subseteq \cB$
$$\sum_{A\subseteq C\subseteq \cB\atop D\subseteq C} (-1)^{|C-A|}q^{|\cB-C|}= (q-1)^{|\cB-(A\cup D)|} (-1)^{|(\cB-A)\cap D|}.$$

\subsection{Tensor products} \label{TensorProducts}

For any supercharacter theory $\tS$ of $G$, the space $\f(G;\Cl(\tS))$
is closed under point-wise multiplication.  However, as with irreducible characters it is often not obvious how to decompose the point-wise product of two supercharacters into supercharacters (meanwhile the superclass identifier functions are orthogonal idempotents).  However, this problem seems far more tractable for normal lattice supercharacter theories.

For $M\in \ker(\tS)$ let $\cC(M)$ be its set of covers.  The following lemma is framed in the language of groups, but applies to arbitrary modular lattices.

\begin{lemma} \label{CoverMeetLemma}
For $M,N\in \ker(\tS)$, 
$$\overline{\cC(M)}\cap \overline{\cC(N)}= \overline{\cC(M\cap N)}.$$
\end{lemma}

\begin{proof}
First, suppose $O\in\cC(M\cap N)$.  Then either $O\subseteq M$ or $OM$ covers $M$, so in either case $O\subseteq \overline{\cC(M)}$.  Similarly $O\subseteq \overline{\cC(N)}$.  Thus, $\overline{\cC(M\cap N)}\subseteq \overline{\cC(M)}\cap \overline{\cC(N)}$.

Next we want to show that $\overline{ \cA}\cap\overline{ \cB}\subseteq \overline{\cC(M\cap N)} $ for all subsets $\cA\subseteq \cC(M)$ and $\cB\subseteq \cC(N)$.  We will use double induction on $|\cA|$ and $|\cB|$.  If $|\cB|=1$, then let $\{O\}=\cB$ so $\overline{\cB}=O$.   Since $O$ covers $N$ we have that either $(O\cap  \overline{\cA})N=O$ or $(O\cap  \overline{\cA})N=N$.

If $(O\cap  \overline{\cA})N=O$ and $M=M\cap N$, then   
$$(O\cap  \overline{\cA})\subseteq   \overline{\cA}\subseteq \overline{\cC(M)}=\overline{\cC(M\cap N)}.$$  
Otherwise, $M\cap O$ covers $M\cap N$.  In this case, $N\cap  \overline{\cA}$ is generated by covers of $M\cap N$ contained in $N$ and since $O\cap  \overline{\cA}$ covers $N \cap \overline{\cA}$, we have 
$$O\cap  \overline{\cA}=(N\cap \overline{\cA})(O\cap M)\subseteq \overline{\cC(M\cap N)}.$$

If $(O\cap  \overline{\cA})N=N$, then $\overline{\cA}\cap \overline{\cB}\subseteq N$ and by modularity, the covers of $M\cap N$ in $N$ generate $\overline{\cA}\cap \overline{\cB}$.  By symmetry, the result also holds if $|\cA|=1$ and $\cB$ is arbitrary.

Suppose $|\cA|,|\cB|>1$.  If there exists $O\in \cB$ such that 
$$\overline{\cB}=\overline{\cB-\{O\}},$$
then by induction $\overline{\cA}\cap \overline{\cB} = \overline{\cA}\cap \overline{\cB-\{O\}} \subseteq \overline{\cC(M\cap N)}$.  Thus, WLOG we may assume that 
$$\overline{\cB}\neq \overline{\cB-\{O\}}\quad\text{and by symmetry}\quad \overline{\cA}\neq \overline{\cA-\{P\}}$$
for all $O\in \cB$, $P\in \cA$.  

As before $(\overline{\cA}\cap \overline{\cB}) \overline{\cB-\{O\}}\in \{ \overline{\cB-\{O\}}, \overline{\cB}\}$.  If  $(\overline{\cA}\cap \overline{\cB}) \overline{\cB-\{O\}}= \overline{\cB-\{O\}}$, then we are done by induction.  By symmetry, if 
$(\overline{\cA}\cap \overline{\cB}) \overline{\cA-\{P\}}= \overline{\cA-\{P\}}$, then we are done by induction. 
We therefore have
$$
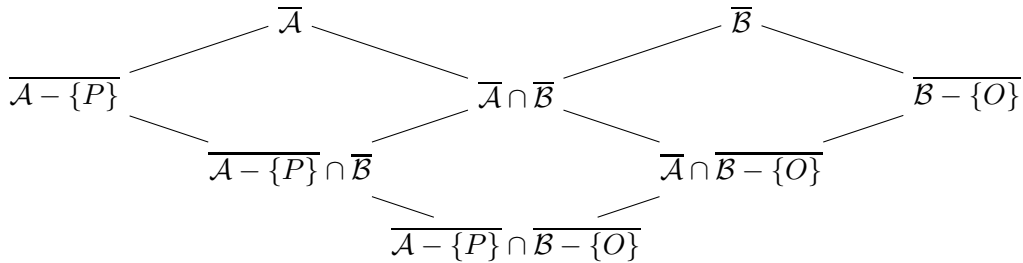
\begin{tikzpicture}
\foreach \l/\x/\y/\e in {20/0/0/{\overline{\cA-\{P\}}},10/3/-1/{\overline{\cA-\{P\}}\cap \overline{\cB}},00/6/-2/{\overline{\cA-\{P\}}\cap \overline{\cB-\{O\}}},30/3/1/{\overline{\cA}},31/9/1/{\overline{\cB}}, 21/6/0/{\overline{\cA}\cap \overline{\cB}},11/9/-1/{\overline{\cA}\cap \overline{\cB-\{O\}}},22/12/0/\overline{\cB-\{O\}}}
	\node (\l) at (\x,\y) {$\e$};
\draw (11) -- (00) -- (10) -- (20) -- (30) -- (21)  -- (11) -- (22) -- (31) -- (21) -- (10); 
\end{tikzpicture}
$$
If there exists $P\in \cA$ and $O\in \cB$ such that $\overline{\cA-\{P\}}\neq \overline{\cB-\{O\}}$, then since $\overline{\cA}\cap \overline{\cB}$ covers these two groups we have
$$\overline{\cA}\cap \overline{\cB}=\overline{\cA-\{P\}}\ \overline{\cB-\{O\}}\subseteq \overline{\cC(M\cap N)},$$
by induction.  Else, for all $O\in \cB$, $Q\in \cA$
$$\overline{\cB-\{O\}}\cap \overline{\cA}=\overline{\cA-\{Q\}}\cap \overline{\cB}.$$
In particular,
$$\overline{\cB-\{O\}}\cap \overline{\cA}\subseteq \overline{\cA}\cap \bigcap_{Q\in \cB} \overline{\cB-\{Q\}}=\overline{\cA}\cap N.$$
By symmetry 
$$\overline{\cB-\{O\}}\cap \overline{\cA}=\overline{\cA-\{P\}}\cap \overline{\cB}\subseteq M\cap \overline{\cB}.$$
Thus, 
$$\overline{\cB-\{O\}}\cap\overline{ \cA}\subseteq M\cap \overline{\cB}\cap N\cap \overline{\cA}=M\cap N.$$
In this case, since $\overline{\cA}\cap \overline{\cB}$ covers $\overline{\cB-\{O\}}\cap \overline{\cA}$ it must also cover $M\cap N$ and we are done.
\end{proof}

\begin{corollary} \label{TensorProduct}  Suppose $M,N\in \ker(\tS)$ with $\cC(M)$ and $\cC(N)$ in general position over $M$ and $N$, respectively.
 Suppose further that for each $O\in \cC(M\cap N)$ either $O\subseteq M$ or $O\subseteq N$.  Then
$$\frac{\chi^{M^\bullet}\odot \chi^{N^\bullet}}{\chi^{M^\bullet}(1)\chi^{N^\bullet}(1)}=\frac{\chi^{(M\cap N)^\bullet}}{\chi^{(M\cap N)^\bullet}(1)}$$
\end{corollary}
\begin{proof} We first show that $\cC(M\cap N)$ is in general position over $M\cap N$.  Let $O\in \cC(M\cap N)$.  WLOG $ON\in \cC(N)$.  We have that 
$$\overline{\cC(M\cap N)-\{O\}}N=\overline{\{PN\mid P\in \cC(M\cap N), P\neq O\}}\neq \overline{\{PN\mid P\in \cC(M\cap N)\}}=\overline{\cC(M\cap N)}N,$$
since $\{PN\mid P\in \cC(M\cap N), PN\in \cC(N)\}\subseteq \{PN\mid P\in \cC(M\cap N)\}$ must also be in general position over $N$.   However then,
$$\overline{\cC(M\cap N)-\{O\}}\neq \overline{\cC(M\cap N)}.$$

By Corollary \ref{CharacterFormula} and Lemma \ref{CoverMeetLemma},
$$\Big(\chi^{M^\bullet}\odot \chi^{N^\bullet}\Big)(g)=0\quad \text{if and only if} \quad \chi^{(M\cap N)^\bullet}(g)=0.$$
Suppose $g\in \overline{\cC(M)}\cap \overline{\cC(N)}=\overline{\cC(M\cap N)}$.  Then there exists subsets $\cA\subseteq\cC(M)$, $\cB\subseteq \cC(N)$, and $\cI\subseteq \cC(M\cap N)$ such that $g\in \overline{\cA}_\circ\cap \overline{\cB}_\circ$ and $\overline{\cA}\cap \overline{\cB}=\overline{\cI}$. By assumption,
$$\cI=\cI_M\sqcup \cI_N,\quad \text{where} \quad \cI_K=\{O\in \cI\mid O\subseteq K\}.$$
So if $g\in \overline{\cI}$, then 
$$g\in M\cdot \overline{\cI_N}=\prod_{O\in \cI_N}MO=\overline{\cA}.$$
Thus, $P\in \cA$ implies $P\cap N\neq M\cap N$ and similarly $Q\in \cB$ implies $Q\cap M\neq M\cap N$.  Suppose $P,P'\in \cA$.  Then $P\cap N=P'\cap N$ implies $(P\cap P')\cap N=(P\cap N)\cap (P'\cap N)=P\cap N$.  We conclude $P=P'$, and similarly, for $Q,Q'\in \cB$, $Q\cap M=Q'\cap M$ if and only if $Q=Q'$.  We therefore have bijections
$$\begin{array}{ccc} \cI_M & \longrightarrow & \cB\\ O & \mapsto & ON\\ Q\cap M & \mapsfrom & Q\end{array} \quad \text{and}\quad  \begin{array}{ccc} \cI_N & \longrightarrow & \cA\\ O & \mapsto & OM\\ P\cap N & \mapsfrom & P\end{array}\ .$$

 Then Corollary \ref{CharacterFormula} implies
\begin{align*}
\chi^{M^\bullet}\odot \chi^{N^\bullet} (g)&= \chi^{M^\bullet}(1)\chi^{N^\bullet}(1)\prod_{P\in \cA\atop Q\in \cB}\frac{1}{(1-|P/M|)(1-|Q/N|)}\\
&= \chi^{M^\bullet}(1)\chi^{N^\bullet}(1)\prod_{O\in \cI_N}\frac{1}{(1-|OM/M|)}\prod_{O\in \cI_N}\frac{1}{(1-|ON/N|)}\\
&= \chi^{M^\bullet}(1)\chi^{N^\bullet}(1)\prod_{O\in \cI_N}\frac{1}{(1-|O/(M\cap N)|)}\prod_{O\in \cI_N}\frac{1}{(1-|O/(M\cap N)|)}\\
&=\frac{\chi^{M^\bullet}(1)\chi^{N^\bullet}(1)}{\chi^{(M\cap N)^\bullet}(1)} \chi^{(M\cap N)^\bullet}(g),
\end{align*}   
as desired.
\end{proof}

\begin{examples} 

\vspace{.25cm}

\noindent\textbf{Cyclics.}
Let $C_n$ be cyclic and suppose $a,b\mid n$.  Fix $C_{p\gcd(a,b)}\in \cC(C_{\gcd(a,b)})$.  Then $C_{p\gcd(a,b)}\subseteq C_a$ if $p\mid \frac{a}{\gcd(a,b)}$ and $C_{p\gcd(a,b)}\subseteq C_b$ if $p\mid\frac{b}{\gcd(a,b)}$.  In other words, the hypotheses of the Corollary \ref{TensorProduct} are satisfied if 
$$\Big\{p\text{ prime}\ \Big|\  p\mid \frac{n}{\gcd(a,b)}\Big\}=\Big\{p\text{ prime}\ \Big|\ p\mid \frac{\lcm(a,b)}{\gcd(a,b)}\Big\}.$$

\vspace{.5cm}

\noindent\textbf{Vector spaces.}  For $\subsp(V)$ the hypotheses are generally not satisfied.  For $\subsp_\cB(V)$, we have that the hypotheses are satisfied if and only if $A,B\in \cB$ satisfy $A\cup B= \cB$.
\end{examples}

\subsection{Restriction formula} \label{RestrictionFormula}

In this section we further assume that 
$$\ker(\tS)=J^\vee(\cI)$$
where $\cI$ are the intersection-irreducible subgroups of $\ker(\tS)$.  
  Thus, for each normal subgroup $N\in \ker(\tS)$, there exists a unique antichain $\cA\in \anti(\cI)$ such that in the notation of (\ref{OverUnderline}),
$$N=\underline{\cA}.$$
Since the dual of a distributive lattice is distributive, we also have that if $\cP$ is the set of product irreducible elements, then for each $N\in\ker(\tS)$ there exists $\cB\subseteq \anti(\cP)$ such that 
$$N=\overline{\cB}.$$ 

\begin{lemma} \label{ProductToCover}
Let $\ker(\tS)$ be a lattice for a supercharacter theory of $G$ and suppose $\ker(\tS)$ is distributive.  Fix $\cB\subseteq \anti(\cP)$.
\begin{enumerate}
\item[(a)] For each $L\in \ker(\tS)$ such that $\overline{\cB}\in \cC(L)$, there exists a unique $K_L\in \cB$ such that $K_L\cap L\neq K_L$.
\item[(b)] The function
$$\begin{array}{ccc} \{L\in \ker(\tS)\mid \overline{\cB}\in \cC(L)\} & \longrightarrow & \cB \\
L & \mapsto & K_L\\ M_K\overline{\cB-K}&\mapsfrom & K\end{array}$$
is a bijection, where $M_K$ is the unique element such that $K\in \cC(M_K)$.
\end{enumerate} 
\end{lemma}
\begin{proof}
(a) Fix $L$ such that $\overline{\cB}\in \cC(L)$. Then there exists $K\in \cB$ such that $K\neq K\cap L$.  Then $KL=\overline{\cB}$ and so $K\in \cC(K\cap L)$.  Suppose there exists $L'$ with $\overline{\cB}\in \cC(L')$ such that $K\neq K\cap L'$.  Then since $K$ is product irreducible $K\cap L'=K\cap L$.  But then $L'=\overline{\cB}\cap L'=KL\cap L'=(K\cap L')(L\cap L')$ which forces $L=L'$. 

(b) The uniqueness of $M_K$ implies that $L\cap K=M_K$ and $M_K\prod_{K\neq K'\in \cB}K'=L$.
\end{proof}

This result allows us to prove some convenient features about distributive lattices of normal subgroups.  

\begin{corollary} Let $\ker(\tS)$ be a lattice for a supercharacter theory of $G$ and suppose $\ker(\tS)$ is distributive. For any normal subgroup $N\in\ker(\tS)$,
\begin{enumerate}
\item[(a)] $N_{\circ}\neq \emptyset$,
\item[(b)] $X^{N^\bullet}\neq \emptyset$,
\item[(c)] $\cC(N)$ is in general position over $N$.
\end{enumerate} 
\end{corollary}

\begin{proof}
(a) For each $L$ such that $N\in \cC(L)$ we can select an element $g_L\in K_L-(K_L\cap L)$ so that $\prod_{N\in \cC(L)} g_L\in N$ but $\prod_{N\in \cC(L)} g_L\notin L'$ for any $L'$ such that $N\in \cC(L')$.  Thus, $\prod_{N\in \cC(L)} g_L\in N_\circ$.

(b) Since $|\ker(\tS)|=|\Cl|=|\Ch|$, we must have that all sets $X^{N^\bullet}$ are nonempty.

(c) We have that $\underline{\cC(N)-\{O\}}\neq \underline{\cC(N)}$ for $O\in \cC(N)$ if and only if $\underline{\cC(N)-\{O\}}\cap O=N$.  But the latter condition follows easily from distributivity.
\end{proof}

Since the dual of a distributive lattice is distributive, we obtain the dual to Lemma \ref{ProductToCover} (which seems somewhat harder to prove directly).

\begin{corollary} \label{CoversToIrreducibles} Let $\ker(\tS)=J^\vee(\cI)$ be a distributive lattice and $\cA\in\anti(\cI)$.  Then
\begin{enumerate}
\item[(a)] For each $O\in \cC(\underline{\cA})$, there exists a unique element $P_O\in \cA$ such that $P_O O\in \cC(P_O)$,
\item[(b)] The function 
$$\begin{array}{c@{\ }c@{\ }c} 
\dd\cC(\underline{\cA}) & \longrightarrow & \cA\\
 O & \mapsto & P_O\\
 \underline{\cC(P)}\cap \underline{\cA-\{P\}} &\mapsfrom & P
\end{array}$$
is a bijection.
\end{enumerate}
\end{corollary}

Let $H\subseteq G$ be a subgroup and suppose $J^\vee(\cI_H)\subseteq \ker(H)$ and $J^\vee(\cI_G)\subseteq \ker(G)$ are distributive lattices.  For the restriction functor 
$$\Res_H^G: \f(G;\Cl(J^\vee(\cI_G))\longrightarrow  \f(H;\Cl(J^\vee(\cI_H))$$
to be well-defined, we minimally require that for each $N\in J^\vee(\cI_G)$, there exists a subset $A\subseteq J^\vee(\cI_H)$ such that 
$$N_\circ\cap H=\bigcup_{M\in A}M_\circ.$$
For the restriction result below, we will want slightly stronger compatibility between the lattices.  In particular, we say  $J^\vee(\cI_G)$ and $J^\vee(\cI_H)$ are \textbf{\emph{restriction favorable}} if
\begin{enumerate}
\item[(R1)]  The function
$$\begin{array}{r@{\ }c@{\ }c@{\ }c}\cdot\cap H:&  J^\vee(\cI_G) & \longrightarrow & J^\vee(\cI_H)\\
&N & \mapsto & N\cap H
\end{array}$$
is well-defined (that is, $N\cap H\in J^\vee(\cI_H)$),
\item[(R2)] If $M,N\in J^\vee(\cI_G)$ with $N\in \cC(M)$, then either $M\cap H=N\cap H$ or $N\cap H\in \cC(M\cap H)$,
\end{enumerate}

\begin{remark}
At first glance it seems that (R2) should always hold due to the diamond isomorphism theorem.  In fact, if we use the full lattice of normal subgroups for both groups this is the case.  However, (R2) guarantees that the lattice for $G$ is not to coarse with respect to the lattice of $H$.
\end{remark}

\begin{examples}\hfill

\vspace{.25cm}

\noindent\textbf{Cyclics.}
For $C_m\subseteq C_n$, we have $\cdot \cap C_m: \ker(C_n)\longrightarrow \ker(C_m)$  is well-defined, since $\ker(C_m)$ is in fact an interval in $\ker(C_n)$.  The other condition also follow easily, using (\ref{CyclicIrreducibles}) and modularity of the lattice.

\vspace{.5cm}

\noindent\textbf{Vector spaces.}  For $\subsp_\cB(V)$, let $\cA\subseteq \cB$.  Then $U=\FF_q\spanning\{a\in \cA\}$ has lattice $\subsp_\cA(U)$ an interval in $\subsp_\cB(V)$.  So as with the cyclics case, $V$ and $U$ are restriction favorable.
\end{examples}

Let
\begin{equation*}
\cA_{H}=\{P_{O\cap H}\mid O\in \cC(\underline{\cA}), O\cap H\neq \underline{\cA}\cap H\}\subseteq \cI_H.
\end{equation*}

\begin{theorem} \label{RestrictionTheorem} Let $H\subseteq G$ be a subgroup, and suppose $J^\vee(\cI_G)\subseteq \ker(G)$ and $J^\vee(\cI_H)\subseteq \ker(H)$ are restriction favorable.  For $\cA\in\anti(\cI_G)$,
\begin{enumerate}
\item[(a)]  The restriction of $\chi^{\underline{\cA}^\bullet}$ factors as
$$\frac{\Res^G_H(\chi^{\underline{\cA}^\bullet})}{\chi^{\underline{\cA}^\bullet}(1)}=\frac{\chi^{\underline{\cA_H}^\bullet}}{\chi^{\underline{\cA_H}^\bullet}(1)}\odot \frac{\chi^{\overline{\cC(\underline{\cA})}\cap H}}{\chi^{\overline{\cC(\underline{\cA})}\cap H}(1)}.$$
\item[(b)] The restriction of $\chi^{\underline{\cA}^\bullet}$  decomposes as
$$\frac{\Res^{G}_{H}(\chi^{\underline{\cA}^{\bullet}})}{\chi^{\underline{\cA}^\bullet}(1)}=\sum_{\underline{\cA_H} \supseteq K\supseteq \overline{\cC(\underline{\cA})}\cap\underline{\cA_H}}\frac{|\overline{\cC(\underline{\cA})}\cap \underline{\cA_H} |(-1)^{|\{Q\in \cC(K)\mid Q\subseteq \underline{\cA_H}\}|}}{|\overline{\cC(K)}\cap \underline{\cA_H}|\chi^{K^\bullet}((\overline{\cC(K)}\cap \underline{\cA_H})_\circ)}\chi^{K^\bullet},$$
where all the terms have nonzero coefficients.
\end{enumerate}
\end{theorem}

Before proving the theorem, we will first need a lemma that computes degree sums using Theorem \ref{DegreeSumTheorem}.

\begin{lemma}\label{RestrictionCoefficient}
Let $H\subseteq G$ be a subgroup, and suppose $J^\vee(\cI_G)\subseteq \ker(G)$ and $J^\vee(\cI_H)\subseteq \ker(H)$ are restriction favorable.  Let $\cA\in\anti(\cI_G)$ and $\underline{\cA_H}\supseteq K\supseteq \overline{\cC(\underline{\cA})}\cap \underline{\cA_H}$.  Then
$$\frac{1}{\chi^{\overline{\cC(\underline{\cA})}\cap H}(1)}\Big( \sum_{N\supseteq \overline{\cC(\underline{\cA})}\cap H\atop N\cap \underline{\cA_H}=K  } \chi^{N^\bullet}(1)\Big)\frac{1}{\chi^{K^\bullet}(1)}=\frac{|\overline{\cC(\underline{\cA})}\cap \underline{\cA_H} |(-1)^{|\{Q\in \cC(K)\mid Q\subseteq \underline{\cA_H}\}|}}{|\overline{\cC(K)}\cap \underline{\cA_H}|\chi^{K^\bullet}((\overline{\cC(K)}\cap \underline{\cA_H})_\circ)}.$$
\end{lemma}

\begin{proof} Suppose $N\cap \underline{\cA_H}=K$ for all $N\supseteq K(\overline{\cC(\underline{\cA})}\cap H)$.  Then $N=H$ implies $K=H\cap \underline{\cA_H}=\underline{\cA_H}$.   In this case, by Theorem \ref{DegreeSumTheorem}, 
\begin{align*}\frac{1}{\chi^{\overline{\cC(\underline{\cA})}\cap H}(1)}\Big( \sum_{N\supseteq \overline{\cC(\underline{\cA})}\cap H\atop N\cap \underline{\cA_H}=\underline{\cA_H}  } \chi^{N^\bullet}(1)\Big)\frac{1}{\chi^{\underline{\cA_H}^\bullet}(1)}&=\frac{|\overline{\cC(\underline{\cA})}\cap H|}{|H|}\frac{|H|}{|\underline{\cA_H}(\overline{\cC(\underline{\cA})}\cap H)|}\frac{1}{\chi^{\underline{\cA_H}^\bullet}(1)}\\ &=\frac{|\overline{\cC(\underline{\cA})}\cap \underline{\cA_H}|}{|\underline{\cA_H}|}\frac{(-1)^0}{\chi^{\underline{\cA_H}^\bullet}(\underline{\cA_H}_\circ)}\\
&=\frac{|\overline{\cC(\underline{\cA})}\cap \underline{\cA_H}|}{|\overline{\cC(\underline{\cA_H})}\cap \underline{\cA_H}|}\frac{(-1)^{|\{Q\in \cC(\underline{\cA_H}\})\mid Q\subseteq \underline{\cA_H}\}|}}{\chi^{\underline{\cA_H}^\bullet}((\overline{\cC(\underline{\cA_H})}\cap \underline{\cA_H})_\circ)}.
\end{align*}
The other terms all satisfy $N\cap\underline{\cA_H}\neq K$ for some $N\supseteq (\overline{\cC(\underline{\cA})}\cap H)$.

 By Theorem \ref{DegreeSumTheorem} and Corollary \ref{SuperDegreeFormula},
\begin{align*}
\frac{1}{\chi^{\overline{\cC(\underline{\cA})}\cap H}(1)}&\Big( \sum_{N\supseteq \overline{\cC(\underline{\cA})}\cap H\atop N\cap \underline{\cA_H}=K  } \chi^{N^\bullet}(1)\Big)\frac{1}{\chi^{K^\bullet}(1)}\\&=\frac{|\overline{\cC(\underline{\cA})}\cap H||H|}{|H||\overline{\cC_{\underline{\cA_H}}^\perp(K,\overline{\cC(\underline{\cA})}\cap H)}|} \frac{\dd\prod_{O\in \cC_{\underline{\cA_H}}^\perp(K,\overline{\cC(\underline{\cA})}\cap H)} \Big(\frac{|O|}{|K(\overline{\cC(\underline{\cA})}\cap H)|}-1\Big)}{\dd\prod_{Q\in \cC(K)} \Big(\frac{|Q|}{|K|}-1\Big)}\frac{|\overline{\cC(K)}|}{|H|}.
\end{align*}
However, for every $O\in \cC_{\underline{\cA_H}}^\perp(K,\overline{\cC(\underline{\cA})}\cap H)$, $O\cap \underline{\cA_H}\in \cC(K)$ with 
$$\frac{|O|}{|K(\overline{\cC(\underline{\cA})}\cap H)|}=\frac{|O\cap \underline{\cA_H}|}{|K|},$$
and every cover of $K$ sitting in $\underline{\cA_H}$ appears in this way.  Thus,
$$\frac{1}{\chi^{\overline{\cC(\underline{\cA})}\cap H}(1)}\Big( \sum_{N\supseteq \overline{\cC(\underline{\cA})}\cap H\atop N\cap \underline{\cA_H}=K  } \chi^{N^\bullet}(1)\Big)\frac{1}{\chi^{K^\bullet}(1)}=\frac{|\overline{\cC(\underline{\cA})}\cap H|}{|\overline{\cC_{\underline{\cA_H}}^\perp(K,\overline{\cC(\underline{\cA})}\cap H)}|}\prod_{Q\in \cC(K)\atop Q\nsubseteq \underline{\cA_H}} \Big(\frac{|Q|}{|K|}-1\Big)^{-1}\frac{|\overline{\cC(K)}|}{|H|}.$$
Also,
$$\frac{|\overline{\cC_{\underline{\cA_H}}^\perp(K,\overline{\cC(\underline{\cA})}\cap H)}|}{|K(\overline{\cC(\underline{\cA})}\cap H)|}=\frac{|\overline{\{Q\in \cC(K)\mid Q\subseteq \underline{\cA_H}\}}|}{|K|},$$
and by distributivity
$$\overline{\{Q\in \cC(K)\mid Q\subseteq \underline{\cA_H}\}}\cap \overline{\{Q\in \cC(K)\mid Q\nsubseteq \underline{\cA_H}\}}=K,$$
so
\begin{align*}
\frac{|\overline{\cC(K)}|}{|\overline{\cC_{\underline{\cA_H}}^\perp(K,\overline{\cC(\underline{\cA})}\cap H)}|}&=\frac{|\overline{\cC(K)}|}{|\overline{\{Q\in \cC(K)\mid Q\subseteq \underline{\cA_H}\}}|}\frac{{|\overline{\{Q\in \cC(K)\mid Q\subseteq \underline{\cA_H}\}}|}}{|\overline{\cC_{\underline{\cA_H}}^\perp(K,\overline{\cC(\underline{\cA})}\cap H)}|}\\
&=\frac{|\overline{\cC(K)}|}{|\overline{\{Q\in \cC(K)\mid Q\subseteq \underline{\cA_H}\}}|}\frac{|K|}{
|K(\overline{\cC(\underline{\cA})}\cap H)|}\\
&=\frac{|\overline{\cC(K)}||\overline{\cC(\underline{\cA})}\cap \underline{\cA_H}|}{|\overline{\{Q\in \cC(K)\mid Q\subseteq \underline{\cA_H}\}}|
|(\overline{\cC(\underline{\cA})}\cap H)|}.
\end{align*}
Since 
$$\overline{\cC(K)}\cap \underline{\cA_H}=\Big(\prod_{Q\in \cC(K)\atop Q\nsubseteq \cA_H} K\Big)\Big(\prod_{Q\in \cC(K)\atop Q\subseteq \cA_H} Q\Big)\cap \underline{\cA_H}=\overline{\{Q\in \cC(K)\mid Q\subseteq \underline{\cA_H}\}},$$
we have
\begin{align*}
\frac{1}{\chi^{\overline{\cC(\underline{\cA})}\cap H}(1)}\Big( \sum_{N\supseteq \overline{\cC(\underline{\cA})}\cap H\atop N\cap \underline{\cA_H}=K  } \chi^{N^\bullet}(1)\Big)\frac{1}{\chi^{K^\bullet}(1)}&=\frac{|\overline{\cC(K)}||\overline{\cC(\underline{\cA})}\cap \underline{\cA_H}|}{|H||\overline{\cC(K)}\cap \underline{\cA_H}|}\prod_{Q\in \cC(K)\atop Q\nsubseteq \underline{\cA_H}} \Big(\frac{|Q|}{|K|}-1\Big)^{-1}\\
&=\frac{|\overline{\cC(\underline{\cA})}\cap \underline{\cA_H} |(-1)^{|\{Q\in \cC(K)\mid Q\subseteq \underline{\cA_H}\}|}}{|\overline{\cC(K)}\cap \underline{\cA_H}|\chi^{K^\bullet}((\overline{\cC(K)}\cap \underline{\cA_H})_\circ)},
\end{align*}
as desired.
\end{proof}
\begin{proof}[Proof of Theorem \ref{RestrictionTheorem}] 
(a) Note that by definition $\underline{\cA_H} \supseteq \underline{\cA}\cap H$, and since covers of $\underline{\cA}\cap H$ either are in $\underline{\cA_H} $ or generate covers of $\underline{\cA_H}$, we conclude that $\overline{\cC(\underline{\cA_H})}\supseteq \overline{\cC(\underline{\cA}\cap H)}$.  Thus, both sides are 0 if and only if $g\notin\overline{\cC(\underline{\cA})}\cap H$.

 Let $g\in \overline{\cB}\cap H$ for $\cB\subseteq \cC(\underline{\cA})$ minimal.  If $B\in \cB$ such that $B\cap H=\underline{\cA}\cap H$, then by distributivity, $g\in \overline{\cB-\{B\}}\cap H$.  Thus, we may assume that $g\in  \overline{\cB} \cap H$ where $\cB\subseteq \{O\in \cC(\underline{\cA})\mid O\cap H\neq \underline{\cA}\cap H\}.$  Then
 \begin{align*}
\Big(\frac{\chi^{\underline{\cA_H}^\bullet}}{\chi^{\underline{\cA_H}^\bullet}(1)}\odot \frac{\chi^{\overline{\cC(\underline{\cA})}\cap H}}{\chi^{\overline{\cC(\underline{\cA})}\cap H}(1)}\Big)(g)&=\prod_{O\in \cB} \frac{1}{1-|(O\cap H)\underline{\cA_H}/\underline{\cA_H}|}\\
&=\prod_{O\in \cB}\frac{1}{1-|(O\cap H)/(\underline{\cA}\cap H)|}\\
&=\prod_{O\in \cB}\frac{1}{1-|O/\underline{\cA}|}\\
&=\frac{\Res^G_H(\chi^{\underline{\cA}^\bullet})}{\chi^{\underline{\cA}^\bullet}(1)} (g).
\end{align*}

(b) We next want to use apply Corollary \ref{TensorProduct}.  Thus, we need to show that for $N\supseteq \overline{\cC(\underline{\cA})}\cap H$ and $O\in \cC(N\cap \underline{\cA_H})$, we have $O\subseteq N$ or $O\subseteq \underline{\cA_H}$.  Since $O$ covers $N\cap \underline{\cA_H}$ we have $O\underline{\cA_H}=\underline{\cA_H}$ or $O\underline{\cA_H}$ covers $\underline{\cA_H}$.  In the first case, $O\subseteq \underline{\cA_H}$.  In the second, $O\underline{\cA_H}$ corresponds to  a unique cover $O'$ of $\underline{\cA}\cap H$, with $O'\underline{\cA_H}=O\underline{\cA_H}$.  But then $O'(N\cap\underline{\cA_H})=O$ and $O\subseteq N$ (since $O'\subseteq \overline{\cC(\underline{\cA})}\cap H\subseteq N$). 

Therefore we can apply  Corollary \ref{TensorProduct} to our situation to obtain
\begin{align*}
\frac{\Res^G_H(\chi^{\underline{\cA}^\bullet})}{\chi^{\underline{\cA}^\bullet}(1)}&=\frac{\chi^{\underline{\cA_H}^\bullet}}{\chi^{\underline{\cA_H}^\bullet}(1)}\odot \frac{\chi^{\overline{\cC(\underline{\cA})}\cap H}}{\chi^{\overline{\cC(\underline{\cA})}\cap H}(1)}\\
&=\frac{\chi^{\underline{\cA_H}^\bullet}}{\chi^{\underline{\cA_H}^\bullet}(1)\chi^{\overline{\cC(\underline{\cA})}\cap H}(1)}\odot\sum_{N\supseteq \overline{\cC(\underline{\cA})}\cap H} \chi^{N^\bullet}\\
&=\frac{1}{\chi^{\overline{\cC(\underline{\cA})}\cap H}(1)}\sum_{N\supseteq \overline{\cC(\underline{\cA})}\cap H} \chi^{N^\bullet}(1)\frac{\chi^{(N\cap \underline{\cA_H})^\bullet}}{\chi^{(N\cap \underline{\cA_H})^\bullet}(1)}.
\end{align*}
Reorganizing,
\begin{equation*}
\frac{\Res^G_H(\chi^{\underline{\cA}^\bullet})}{\chi^{\underline{\cA}^\bullet}(1)}=\frac{1}{\chi^{\overline{\cC(\underline{\cA})}\cap H}(1)}\sum_{\underline{\cA_H}\supseteq K\supseteq\overline{\cC(\underline{\cA})}\cap \underline{\cA_H}}\Big( \sum_{N\supseteq \overline{\cC(\underline{\cA})}\cap H\atop N\cap \underline{\cA_H}=K  } \chi^{N^\bullet}(1)\Big)\frac{\chi^{K^\bullet}}{\chi^{K^\bullet}(1)}.
\end{equation*}
Note that 
$$K(\overline{\cC(\underline{\cA})}\cap H)\cap \underline{\cA_H}=(K\cap \underline{\cA_H})(\overline{\cC(\underline{\cA})}\cap \underline{\cA_H})=K(\overline{\cC(\underline{\cA})}\cap \underline{\cA_H}),$$
which equals $K$ if and only if $\overline{\cC(\underline{\cA})}\cap \underline{\cA_H}\subseteq K$.    Thus, we can apply Lemma \ref{RestrictionCoefficient} to deduce the result.
\end{proof}

\begin{corollary} \label{CollapsedCase} Under the hypotheses of Theorem \ref{RestrictionFormula},
if $\underline{\cA_H}\subseteq \overline{\cC(\underline{\cA})}$, then
$$\frac{\Res^{G}_{H}(\chi^{\underline{\cA}^{\bullet}})}{\chi^{\underline{\cA}^\bullet}(1)}=\frac{\chi^{\underline{\cA_H}^{\bullet}}}{\chi^{\underline{\cA_H}^\bullet}(1)}.$$
\end{corollary}

\begin{example}\hfill

\vspace{.25cm}

\noindent\textbf{Cyclics.}
Let $C_m\subseteq C_n$ and suppose $d\mid n$.   Let
$$P=\{p\text{ prime}\mid pd\mid n\}\quad \text{and}\quad d^\vee=\prod_{p\in P}p.$$
If $n=p_1^{j_1}\cdots p_\ell^{j_\ell}$ and $d=p_1^{i_1}\cdots p_\ell^{i_\ell}$, then
$$C_d=\bigcap_{1\leq k\leq \ell\atop i_k<j_k} C_{p_k^{i_k-j_k}n},\quad\text{so}\quad \cA=\{C_{p_k^{i_k-j_k}n}\mid 1\leq k\leq \ell, i_k<j_k\},$$
and for $i_k\neq j_k$,
$$C_{p_k^{i_k-j_k}n}=\gco(C_{p_kd}),$$
and 
$$\overline{\cC(\underline{\cA})}=\prod_{p\in P} C_{pd}.$$
If $m=p_1^{h_1}\cdots p_\ell^{h_\ell}$, then
$$\cA_{C_m}=\{C_{\gcd(p_k^{i_k-j_k}n,m)}\mid 1\leq k\leq \ell, h_k>i_k\} \quad \text{and}\quad \underline{\cA_{C_m}}=C_{\gcd(d,m)}.$$
Thus,
$$\overline{\cC(\underline{\cA})}\cap \underline{\cA_{C_m}}= \prod_{p\in P} C_{\gcd(d,m)}=\underline{\cA_{C_m}},$$
so
\begin{equation*}
\frac{\Res_{C_m}^{C_n}(\chi^d)}{\chi^d(1)} =\frac{\chi^{\gcd(d,m)}}{\chi^{\gcd(d,m)}(1)}.
\end{equation*}

\vspace{.5cm}

\noindent\textbf{Vector spaces.}  Let $\cB$ be a basis of $V$ and suppose $U\subseteq V$ is subspace with a basis constructed as follows.  There exists a subset $A\subseteq \cB$ and a set partition $\bl(A)$ of $A$ such that 
$$\{\sum_{b\in B}b\mid B\in \bl(A)\}$$
is a basis of $U$.  Then $\cdot \cap U:\subsp_\cB(V)\longrightarrow \subsp_{\bl(A)}(U)$ is a well-defined function, and covers get sent to covers.  However, since $\overline{\cC(W)}=V$ for any subspace of $W$, we have that the hypotheses of Corollary \ref{CollapsedCase} are satisfied.  Thus,
$$\frac{\Res_{U}^{V}(\chi^{\cB-\{b_1,\ldots, b_\ell\}})}{\chi^{\cB-\{b_1,\ldots, b_\ell\}}(1)}=\frac{\chi^{\{\sum_{b\in B}b\mid |B\cap \{b_1,\ldots,b_\ell\}|\neq 1\}}}{\chi^{\{\sum_{b\in B}b\mid |B\cap \{b_1,\ldots,b_\ell\}|\neq 1\}}(1)}.$$
\end{example}

\vspace{2cm}
\noindent (Farid Aliniaeifard) Department of Mathematics, University of Colorado \textbf{Boulder}\\
\textsf{farid.aliniaeifard@colorado.edu}\\
\\
\noindent (Nathaniel Thiem) Department of Mathematics, University of Colorado \textbf{Boulder}\\
\textsf{thiemn@colorado.edu}

\end{document}